\documentclass[12pt]{amsart}

\newtheorem{THM}{Theorem}
\newtheorem{LMA}[THM]{Lemma}
\newtheorem{PROP}[THM]{Proposition}

\numberwithin{equation}{section}

\usepackage{amsmath}
\usepackage{amssymb}
\usepackage{euscript}

\usepackage[pdftex]{graphicx}
\newcommand{\pix}[2]{\begin{center}
{\includegraphics[width=#1\textwidth]{#2}}
\end{center}}

\newcommand{\showon}{\begin{eqnarray*}}
\newcommand{\showoff}{\end{eqnarray*}}

\newcommand{\drop}{\smallsetminus}

\newcommand{\F}{\EuScript{F}}

\newcommand{\M}{\EuScript{M}}

\renewcommand{\P}{\EuScript{P}}

 \newcommand{\RR}{\mathbb{R}}

\newcommand{\V}{\EuScript{V}}
\newcommand{\W}{\EuScript{W}}

 \newcommand{\y}{\mathbf{y}}

\begin{document}

\title[Half-plane property matroids]{A criterion for the half-plane property}

\author{David G. Wagner}
\thanks{Research supported by NSERC Discovery Grant OGP0105392.}
\author{Yehua Wei}
\thanks{Research supported by an NSERC Summer Undergraduate Research Award.}
\address{Department of Combinatorics and Optimization\\
University of Waterloo\\
Waterloo, Ontario, Canada\ \ N2L 3G1}
\email{\texttt{dgwagner@math.uwaterloo.ca}}
\email{\texttt{y4wei@student.math.uwaterloo.ca}}


\begin{abstract}
We establish a convenient necessary and sufficient condition for a
multiaffine real polynomial to be stable, and use it to verify that
the half-plane property holds for seven small matroids that resisted the
efforts of Choe, Oxley, Sokal, and Wagner \cite{COSW}.
\end{abstract}

\maketitle

In recent years, matroid theory has found connections with certain
analytic properties of real multivariate polynomials.
These properties are abstractions of physical characteristics of an electrical network.  Not all matroids exhibit the same physically sensible behaviour that
graphs do.  It is an interesting (and often challenging) problem to determine 
whether a given matroid satisfies one or another of these physically-motivated
conditions.  In this paper we deduce a convenient necessary and sufficient
criterion (Theorem 3(c)) for the ``strong Rayleigh property'', and use it to
verify this property for some small matroids, among them the V\'amos cube $\V_8$.
This supplements Br\"and\'en's result \cite{Br} that the strong Rayleigh
property is equivalent to the ``half-plane property'', and resolves some
questions left open by Choe, Oxley, Sokal, and Wagner \cite{COSW}.\\

Let $Z(y_1,...,y_m)$ be a polynomial with real coefficients, and let
$E=\{1,...,m\}$.  The polynomial $Z$ has the \emph{half-plane property (HPP)}
or is \emph{Hurwitz stable} provided that
whenever $\mathrm{Re}(y_e)>0$ for all $e\in E$ then $Z(y_1,...,y_m)\neq 0$.
The polynomial $Z$ is \emph{stable} provided that
whenever $\mathrm{Im}(y_e)>0$ for all $e\in E$ then $Z(y_1,...,y_m)\neq 0$.
Note that a homogeneous polynomial $Z$ is stable if and only if
it is Hurwitz stable.  If every variable $y_e$ for $e\in E$ occurs in $Z$ to
at most the first power, then the polynomial $Z$ is \emph{multiaffine}.
For any index $e\in E$ let
$$Z_e = \frac{\partial}{\partial y_e}Z$$ 
be the \emph{contraction of $e$ in $Z$}, and let
$$Z^e = \left. Z\right|_{y_e=0}$$
be the \emph{deletion of $e$ from $Z$}.  This notation is extended to
multiple (distinct) sub- and super-scripts in the obvious way.
For distinct indices  $e,f\in E$ let
$$\Delta Z\{e,f\} = Z_e Z_f - Z_{ef} Z.$$
We refer to this as the \emph{Rayleigh difference of $\{e,f\}$ in $Z$}.
If $Z$ is multiaffine then $Z = Z^e + y_e Z_e$ for any $e\in E$, and
it follows that
$$\Delta Z\{e,f\} = Z_e^f Z_f^e - Z_{ef} Z^{ef}.$$
The polynomial $Z$ has the \emph{Rayleigh property} provided that
whenever $y_c>0$ for all $c\in E$, then for all $\{e,f\}\subseteq E$,
$\Delta Z\{e,f\}\geq 0$.  The polynomial $Z$ has the \emph{strong 
Rayleigh property} provided that whenever $y_c\in\RR$ for all $c\in E$,
then for all $\{e,f\}\subseteq E$, $\Delta Z\{e,f\}\geq 0$.

The motivation for these definitions is discussed in \cite{COSW,Wa1}.
The HPP and stability are studied in \cite{BBL,Br,C1,C2,COSW,Wa1},
and the Rayleigh property is studied in \cite{C1,CW,Wa1,Wa2,Wa3}.
Clearly the strong Rayleigh property implies the Rayleigh property.
In \cite{Br}, Br\"and\'en proves that for multiaffine real polynomials
the strong Rayleigh property is equivalent to stability.
We give another criterion equivalent to these in Theorem 3(c).

Now, let $Z(y_1,...,y_m)$ be a multiaffine real polynomial.
To investigate the strong Rayleigh property,
consider any pair $\{e,f\}$ in $E$, and any third index $g\in E\drop\{e,f\}$.
Since $Z$ is multiaffine, $\Delta Z\{e,f\}$ is at most quadratic in $y_g$.
A short calculation shows that
$$\Delta Z\{e,f\} = A y_g^2 + B y_g + C$$
in which
$$A= \Delta Z_g\{e,f\} = Z_{eg}^f Z_{fg}^e - Z_{efg} Z_g^{ef}$$
and
$$B= Z_e^{fg}Z_{fg}^e + Z_{f}^{eg}Z_{eg}^f - Z_g^{ef}Z_{ef}^g - Z_{efg}Z^{efg}$$
and
$$C= \Delta Z^g\{e,f\} =  Z_{e}^{fg} Z_{f}^{eg} - Z_{ef}^g Z^{efg}.$$
If $Z$ has the strong Rayleigh property then for any real values of
$y_c$ ($c\in E\drop\{e,f,g\}$) the polynomial $\Delta Z\{e,f\}$ is nonnegative
for all real values of $y_g$.  Since this quadratic polynomial in $y_g$
does not change sign, its discriminant is nonpositive:\ $B^2-4AC\leq 0$.
In fact, this discriminant has a surprising feature that can be put to good use.

\begin{PROP}
Let $Z(y_1,...,y_m)$ be a multiaffine polynomial, and let $e,f,g$ be distinct indices.
Then the discriminant of $\Delta Z\{e,f\}$ with respect to $y_g$ is symmetric
under all permutations of the indices $\{e,f,g\}$.
\end{PROP}
\begin{proof}
With notation as in the previous paragraph, one calculates that the
discriminant $B^2 - 4AC$ equals
\showon
& &
  (Z_{e}^{fg} Z_{fg}^{e})^2
 +(Z_{f}^{eg} Z_{eg}^{f})^2
 +(Z_{g}^{ef} Z_{ef}^{g})^2
 +(Z_{efg} Z^{efg})^2\\
& &
 -2 Z_{e}^{fg} Z_{fg}^{e} Z_{f}^{eg} Z_{eg}^{f}
 -2 Z_{e}^{fg} Z_{fg}^{e} Z_{g}^{ef} Z_{ef}^{g}
 -2 Z_{f}^{eg} Z_{eg}^{f} Z_{g}^{ef} Z_{ef}^{g}\\
& &
 -2( Z_{e}^{fg} Z_{fg}^{e} 
   + Z_{f}^{eg} Z_{eg}^{f}
   + Z_{g}^{ef} Z_{ef}^{g} ) Z_{efg} Z^{efg}\\
& &
 +4 Z_{e}^{fg} Z_{f}^{eg} Z_{g}^{ef} Z_{efg}
 +4 Z_{fg}^{e} Z_{eg}^{f} Z_{ef}^{g} Z^{efg}.
\showoff
This is clearly symmetric under all permutations of the indices $\{e,f,g\}$,
as claimed.
\end{proof}

\begin{LMA}
Let $Z(y_1,...,y_m)$ be a multiaffine polynomial with positive coefficients,
and let $e,f,g$ be distinct indices.
Fix real values for all $y_c$ ($c\in E\drop\{e,f,g\}$).  Assume
that $\Delta Z_f\{e,g\}\geq 0$,
that $\Delta Z^f\{e,g\}\geq 0$, and
that $\Delta Z\{e,f\}\geq 0$ for all $y_g\in\RR$.
Then $\Delta Z\{e,g\}\geq 0$ for all $y_f\in\RR$.
\end{LMA}
\begin{proof}
We have $\Delta Z\{e,g\} = A y_f^2 + B y_f + C$ with real coefficients
$A\geq 0$, $B$, and $C\geq 0$.  By Proposition 1,
the discriminant of $\Delta Z\{e,g\}$ with respect to $y_f$ equals
the discriminant of $\Delta Z\{e,f\}$ with respect to $y_g$.
By the hypothesis of the Lemma, this discriminant is nonpositive.
Consequently, $\Delta Z\{e,g\}$ does not change sign for $y_f\in\RR$.
It follows that if either $A=0$	or $C=0$ then $B=0$.
Therefore, $\Delta Z\{e,g\}\geq 0$ for all $y_f\in \RR$, as was to be shown.
\end{proof}

\begin{THM}
Let $Z(y_1,...,y_m)$ be a multiaffine polynomial with positive coefficients.
The following conditions are equivalent:\\
\textup{(a)}\ $Z$ is stable;\\
\textup{(b)}\ $Z$ has the strong Rayleigh property;\\
\textup{(c)}\ for every index $e\in E$, both $Z_e$ and $Z^e$
have the strong Rayleigh property, and for some pair of indices $\{e,f\}
\subseteq E$, $\Delta Z\{e,f\} \geq 0$ for all $y_c\in\RR$
($c\in E\drop\{e,f\})$.
\end{THM}
\begin{proof}
The equivalence of conditions (a) and (b) is Theorem 5.6 of Br\"and\'en
\cite{Br}.
That (b) implies (c) is immediate, since the operations
$Z\mapsto Z_e$ and $Z\mapsto Z^e$ both preserve the strong Rayleigh
property (see Proposition 5.4 of \cite{CW}).  It remains only to show that
(c) implies (b).  Assume the hypothesis of (c).
Then $\Delta Z\{e,f\}\geq 0$ for all $y_c\in\RR$ ($c\in E\drop\{e,f\}$),
by hypothesis.  By one application of Lemma 2, we deduce that
for any $g\in E\drop\{e,f\}$,
$\Delta Z\{e,g\}\geq 0$ for all $y_c\in\RR$ ($c\in E\drop\{e,g\}$).
The cases for $\Delta Z\{f,g\}$ follow by symmetry.
For a pair $\{g,h\}$ disjoint from $\{e,f\}$ we apply Lemma 2 again,
beginning with $\Delta Z\{e,g\}\geq 0$ for all $y_c\in\RR$ ($c\in E\drop\{e,g\}$),
to deduce that $\Delta Z\{g,h\}\geq 0$ for all $y_c\in\RR$ ($c\in E\drop\{g,h\}$).
Thus, $Z$ is strongly Rayleigh.
\end{proof}
The advantage of the criterion (c) is that only one Rayleigh difference needs
to be checked, instead of the $\binom{m}{2}$ that are required \emph{a priori}.
The strong Rayleigh condition on the minors $Z_e$ and $Z^e$ can be assumed
if one is proceeding by induction on $|E|$.  We give a few applications of
this method below.

The above concepts apply to a matroid $\M$ with ground-set $E$ through
its \emph{basis-generating polynomial}
$$M(\y) = \sum_{B} \y^B,$$
in which the sum is over the set of all bases $B$ of $\M$, and
$\y^B = \prod_{e\in B} y_e$.
These polynomials are clearly homogeneous and multiaffine.
Thus, $M(\y)$ has the HPP if and only if it is strongly Rayleigh
(Corollary 5.14 of \cite{Br}).
Notice that $M^e$ is the basis-generating polynomial of the deletion
$\M\drop e$, and that $M_e$ is the basis-generating polynomial of the
contraction $\M/e$.  The Potts model partition function of $\M$ is a
much more informative polynomial to which some of the above conditions
can also be applied -- see \cite{So,Wa3}.

The study of matroids with the HPP was begun in \cite{COSW},
from which we require the following results.\\
$\bullet$\ The class of HPP matroids is closed by taking duals and minors.\\
$\bullet$\ Every matroid with at most six elements is HPP.\\
$\bullet$\ Every matroid with rank (or corank) at most two is HPP.\\
$\bullet$\ Up to duality, there are only four matroids with seven
elements for which the HPP is undetermined:\ these are
$\F_7^{-4}$, $\W^{3+}$, $\W^3+e$, and $\P'_7$ depicted in Figure 1.
They all have rank three.\\
$\bullet$\ The rank three seven-element matroids
$\F_7^{-5}$ and $\P''_7$ are HPP.  (The precise structure of these
matroids is not important here -- see \cite{COSW}.)
\begin{figure}
\pix{1.}{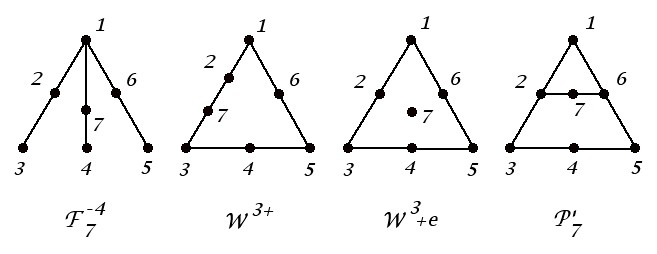}
\caption{The four undetermined cases with $|E|\leq 7$.}
\end{figure}

The HPP remained unsettled in \cite{COSW} for three more small matroids.
These are the V\'amos cube $\V_8$ (of rank 4) and the two one-point deletions
$n\P\drop 1$ and $n\P\drop 9$ (of rank 3) of the non-Pappus matroid $n\P$
depicted in Figure 2.  (The non-Pappus matroid itself is not HPP.)
\begin{figure}
\pix{1.}{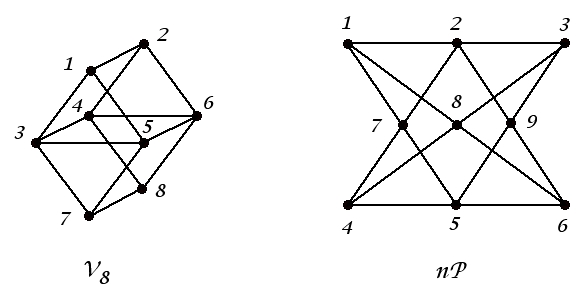}
\caption{The V\'amos cube and the non-Pappus matroid.}
\end{figure}
Using the criterion of Theorem 3(c) we now show that all seven
of these matroids do have the HPP/strong Rayleigh property.
To show that $\Delta M\{e,f\} \geq 0$ for all $y_c\in\RR$ ($c\in E\drop
\{e,f\}$), we write $\Delta M\{e,f\}$ as a sum of squares of polynomials.
(Though not needed for the present purposes, when doing this one should be aware
of the theory developed around Hilbert's 17th problem;\ see \cite{Re},
for example.)\\

\noindent$\bullet$\ \ 
$\M=\F_7^{-4}$:\\
Every one-point deletion or contraction has six elements, so is
strongly Rayleigh by the above remarks.  A short calculation verifies that
\begin{eqnarray*}
& & \Delta M \{ 1,2 \} \\
&=& \frac{1}{2}(y_3y_7+y_5y_7+y_4y_5+y_3y_4+y_3y_5+y_3y_6+y_6y_7+y_4y_6)^2 \\
& & +\frac{1}{2}(y_3y_7+y_3y_4+y_3y_6+y_3y_5)^2 \\
& & +\frac{1}{2}(y_4y_6-y_5y_7)^2 +\frac{1}{2}(y_6y_7-y_4y_5)^2
\end{eqnarray*}
By Theorem 3, it follows that $\F_7^{-4}$ is strongly Rayleigh.\\

\noindent$\bullet$\ \ 
$\M=\W^{3+}$:\\
Every one-point deletion or contraction has six elements, so is
strongly Rayleigh by the above remarks.  A short calculation verifies that
\begin{eqnarray*}
& & \Delta M \{ 1,2 \} \\
&=& \frac{1}{2}(y_6y_3+y_4y_7+y_3y_5+y_5y_7+y_4y_6+y_6y_7+y_4y_3+y_4y_5)^2  \\
& & +\frac{1}{2}(y_6y_3+y_3y_5+y_4y_3+y_4y_5)^2+\frac{1}{2}(y_4y_7+y_5y_7+y_6y_7)^2
    +\frac{1}{2}y_4^2y_6^2 
\end{eqnarray*}
By Theorem 3, it follows that $\W^{3+}$ is strongly Rayleigh.\\

\noindent$\bullet$\ \ 
$\M=\W^3+e$:\\
Every one-point deletion or contraction has six elements, so is
strongly Rayleigh by the above remarks.  A short calculation verifies that
\begin{eqnarray*}
& & \Delta M \{ 1,2 \} \\
&=& \frac{1}{2}(y_3y_7+y_3y_6+y_4y_7+y_4y_6+y_3y_4+y_6y_7+y_5y_7+y_3y_5+y_4y_5)^2 \\
& & +\frac{1}{2}(y_3y_7+y_3y_6+y_3y_4+y_4y_5+y_3y_5)^2+\frac{1}{2}(y_4y_6-y_5y_7)^2 \\
& & +\frac{1}{2}y_4^2y_7^2+\frac{1}{2}y_6^2y_7^2
\end{eqnarray*}
By Theorem 3, it follows that $\W^3+e$ is strongly Rayleigh.\\

\noindent$\bullet$\ \ 
$\M=\P'_7$:\\
Every one-point deletion or contraction has six elements, so is
strongly Rayleigh by the above remarks.  A short calculation verifies that
\begin{eqnarray*}
& & \Delta M \{ 1,2 \} \\
&=& \frac{1}{2}(y_3y_7+y_4y_6+y_4y_7+y_3y_4+y_5y_7+y_5y_3+y_6y_3+y_4y_5)^2 \\
& & +\frac{1}{2}(y_3y_4+y_3y_7+y_5y_3+y_6y_3+y_4y_5)^2 \\
& & +\frac{1}{2}(y_4y_6-y_5y_7)^2 +\frac{1}{2}y_4^2y_7^2
\end{eqnarray*}
By Theorem 3, it follows that $\P'_7$ is strongly Rayleigh.\\

\noindent$\bullet$\ \ 
$\M=n\P\drop 1$:\\
Every one-point contraction has rank $2$, so is strongly Rayleigh.
Every one-point deletion is isomorphic to one of $\F_7^{-4}$,
$\W_3+e$, $\F_7^{-5}$, $\P'_7$, or $\P''_7$.  By the results above,
these are all strongly Rayleigh.  A short calculation verifies that
\begin{eqnarray*}
& & \Delta M \{ 2,4 \} \\
&=& \frac{1}{2}(y_3y_7+y_3y_9+y_5y_9+y_6y_7+y_7y_8+y_7y_9 \\
& & \hspace{1cm} +y_5y_7+y_5y_8+y_6y_8+ y_8y_9+y_3y_6+y_3y_5)^2\\
& & +\frac{1}{2}(y_6y_7+y_7y_8+y_7y_9+y_3y_7+y_5y_7)^2 \\ 
& & +\frac{1}{2}(y_3y_9+y_5y_9+y_3y_5-y_6y_8)^2
    +\frac{1}{2}(y_5y_8-y_3y_6)^2+\frac{1}{2}y_8^2y_9^2
\end{eqnarray*}
By Theorem 3, it follows that $n\P\drop 1$ is strongly Rayleigh.\\

\noindent$\bullet$\ \ 
$\M=n\P\drop 9$:\\
Every one-point contraction has rank $2$, so is strongly Rayleigh.
Every one-point deletion is isomorphic to one of $\F_7^{-4}$ or  $\P'_7$.
By results above, these are both strongly Rayleigh.
A short calculation verifies that
\begin{eqnarray*}
& & \Delta M \{1,2\} \\
&=& \frac{1}{2}(y_4y_6+y_3y_8+y_3y_7+y_6y_7+y_7y_8+y_5y_8 \\
& & \hspace{1cm} +y_3y_6+y_3y_5+y_4y_5+y_5y_6+y_4y_8+y_3y_4)^2 \\
& & +\frac{1}{2}(y_3y_8+y_3y_6+y_3y_7+y_3y_5+y_3y_4+y_4y_8)^2 \\ 
& & +\frac{1}{2}(y_4y_5+y_4y_6+y_5y_6-y_7y_8)^2+\frac{1}{2}(y_5y_8-y_6y_7)^2 \\
\end{eqnarray*}
By Theorem 3, it follows that $n\P\drop 9$ is strongly Rayleigh.\\

\noindent$\bullet$\ \ 
$\M=\V_8$:\\
The V\'amos cube is self-dual.
Every one-point contraction is isomorphic to $\F_7^{-4}$ or  $\F_7^{-5}$.
Every one-point deletion is isomorphic to $(\F_7^{-4})^*$ or  $(\F_7^{-5})^*$.
By the results above, these are all strongly Rayleigh.
A short calculation verifies that
\begin{eqnarray*}
& & \Delta M \{ 1,2 \} \\
&=&\frac{1}{4}(y_3y_4y_5+y_3y_4y_6+y_3y_4y_7+y_3y_4y_8+y_3y_5y_7+y_3y_5y_8 \\
& & \hspace{1cm} +y_3y_6y_7+y_3y_6y_8 +y_4y_5y_7+y_4y_5y_8+y_4y_6y_7 \\
& & \hspace{1cm} +y_4y_6y_8+y_5y_6y_7+y_5y_6y_8+y_5y_7y_8+y_6y_7y_8)^2 \\
& & +\frac{1}{4}(y_3y_4y_5+y_3y_4y_6+y_3y_4y_7+y_3y_4y_8 \\
& & \hspace{1cm} +y_3y_5y_7+y_3y_5y_8+y_4y_6y_7+y_4y_6y_8)^2 \\
& & +\frac{1}{4}(y_3y_4y_5+y_3y_4y_6+y_3y_4y_7+y_3y_4y_8 \\
& & \hspace{1cm} +y_3y_6y_7+y_3y_6y_8+y_4y_5y_8+y_4y_5y_7)^2 \\
& & +\frac{1}{4}(y_3y_4y_5+y_3y_4y_6+y_3y_4y_7+y_3y_4y_8 \\
& & \hspace{1cm} -y_6y_7y_8-y_5y_7y_8-y_5y_6y_7-y_5y_6y_8)^2   \\
& & +\frac{1}{8}(y_3y_6y_7-y_3y_5y_8+y_4y_6y_7-y_4y_5y_8+y_6y_7y_8-y_5y_6y_8)^2 \\
& & +\frac{1}{8}(y_3y_5y_7-y_3y_6y_8-y_4y_6y_8+y_4y_5y_7+y_5y_7y_8-y_5y_6y_8)^2 \\
& & +\frac{1}{8}(y_3y_5y_7+y_3y_6y_7+y_4y_6y_8+y_4y_5y_8+y_5y_6y_7+y_5y_7y_8)^2 \\
& & +\frac{1}{8}(y_3y_5y_8+y_3y_6y_8+y_4y_6y_7+y_4y_5y_7+y_5y_6y_7+y_6y_7y_8)^2 \\
& & +\frac{1}{8}(y_3y_6y_7-y_3y_5y_8+y_4y_6y_7-y_4y_5y_8+y_5y_6y_7-y_5y_7y_8)^2 \\
& & +\frac{1}{8}(y_3y_5y_7-y_3y_6y_8-y_4y_6y_8+y_4y_5y_7+y_5y_6y_7-y_6y_7y_8)^2 \\ 
& & +\frac{1}{8}(y_3y_5y_7+y_3y_6y_7+y_4y_6y_8+y_4y_5y_8+y_5y_6y_8+y_6y_7y_8)^2 \\
& & +\frac{1}{8}(y_3y_5y_8+y_3y_6y_8+y_4y_6y_7+y_4y_5y_7+y_5y_6y_8+y_5y_7y_8)^2 
\end{eqnarray*}
By Theorem 3, it follows that $\V_8$ is strongly Rayleigh.

\end{document}